\documentclass[a4paper,10pt]{amsart}
\usepackage[centertags]{amsmath}
\usepackage[english]{babel}
\usepackage{amsfonts}
\usepackage{amssymb}
\usepackage{amsthm}
\usepackage[usenames,dvipsnames]{color}
\usepackage{newlfont}
\usepackage{graphicx}
\usepackage{epstopdf}
\usepackage{cite}
\usepackage{bm}
\usepackage{blkarray}
\usepackage{multirow}
\usepackage{enumerate}

\newcommand{\xdownarrow}[1]{%
    {\left\downarrow\vbox to #1{}\right.\kern-\nulldelimiterspace}}

\newcommand{\Z}{\mathbb{Z}}

\newcommand{\Q}{\mathbb{Q}}

\newcommand{\bq }{\begin{equation}}
\newcommand{\eq }{\end{equation}}
\theoremstyle{plain}
\newtheorem{thm}{Theorem}[section]

\newtheorem{prop}[thm]{Proposition}

\newtheorem{cor}[thm]{Corollary}
\newtheorem{rem}[thm]{Remark}

\theoremstyle{definition}
\newtheorem{defn}[thm]{Definition}
\newtheorem{example}[thm]{Example}
\DeclareMathOperator{\modu}{{{mod}}}

\theoremstyle{example}

\title{Hypersurface model-fields of definition for smooth hypersurfaces and their twists}

\author[E. Badr] {Eslam Badr}
\address{$\bullet$\,\,Eslam Badr}
\address{Department of Mathematics,
Faculty of Science, Cairo University, Giza-Egypt}
\email{eslam@sci.cu.edu.eg}
\author[F. Bars] {Francesc Bars}
\address{$\bullet$\,\,Francesc Bars Cortina}
\address{Departament Matem\`atiques, Edif. C, Universitat Aut\`onoma de Barcelona\\
08193 Bellaterra, Catalonia} \email{francesc@mat.uab.cat}
\thanks{F.Bars supported by MTM2016-75980-P}

\begin{document}
\maketitle


\begin{abstract} Given a smooth projective variety of dimension $n-1\geq 1$ defined over a perfect field $k$ that admits a non-singular hypersurface modelin $\mathbb{P}^n_{\overline{k}}$ over $\overline{k}$, a fixed algebraic closure of $k$, it does not
necessarily have a non-singular hypersurface model defined over the
base field $k$. We first show an example of such phenomenon: a
variety defined over $k$ admitting non-singular hypersurface models
but none defined over $k$. We also determine under which conditions
a non-singular hypersurface model over $k$ may exist. Now, even
assuming that such a smooth hypersurface model exists, we wonder
about the existence of non-singular hypersurface models over $k$ for
its twists. We introduce a criterion to characterize twists
possessing such models and we also show an example of a twist not
admitting any non-singular hypersurface model over $k$, i.e for any
$n\geq 2$, there is a smooth projective variety of dimension $n-1$ over $k$ which is a twist of a smooth hypersurface variety over $k$, but itself does not admit any non-singular hypersurface model over
$k$. Finally, we obtain a theoretical result to describe all the
twists of smooth hypersurfaces with cyclic automorphism group having
a model defined over $k$ whose automorphism group is generated by a
diagonal matrix.

The particular case $n=2$ for smooth plane curves was studied by the authors jointly with E. Lorenzo Garc\'{\i}a in
\cite{badr2016twists}, and we deal here with the problem in higher dimensions.
\end{abstract}

\begin{section}{Introduction}
Let $X_0,\ldots,X_n$ be a homogenous coordinate system for
$\mathbb{P}^n_{\overline{k}}$, the $n$-dimensional projective space
over $\overline{k}$. Given a smooth projective variety $\overline{V}\subset\mathbb{P}^n_{\overline{k}}$, the group of birational transformations of $V$ onto itself is denoted by $\operatorname{Bir}(\overline{V})$, the group of automorphisms of $V$ (i.e. the group of biregular transformations of $V$ onto itself) is  $\operatorname{Aut}(\overline{V})$, and by $\operatorname{Lin}(\overline{V})$ we mean the subgroup of automorphisms of $\overline{V}$ induced by
projective linear transformations in
$\operatorname{Aut}(\mathbb{P}^n_{\overline{k}})=\operatorname{PGL}_{n+1}(\overline{k})$, the general linear group of $(n+1)\times(n+1)$ projective matrices.

Now, let $\overline{V}$ be a smooth hypersurface in $\mathbb{P}^n_{\overline{k}}$, that is, an $(n-1)$-dimensional smooth projective variety identified with a hypersurface model $H_{\overline{V},d,n}$ represented by a single homogenous polynomial equation, say
$F(X_0,\ldots,X_n)=0$ of some degree $d$ over
$\overline{k}$ without singularities (assume once and for all that $d\geq4$). It is known that a smooth plane curve $\overline{V}$ of degree $d\geq4$ has finitely many automorphisms, and moreover any automorphism is induced by a projective linear transformation of $\mathbb{P}^2_{\overline{k}}$, thus $\operatorname{Aut}(\overline{V})=\operatorname{Lin}(\overline{V})$. Matsumura-Monsky (1946) showed that, for $n\geq3$,
$\operatorname{Lin}(\overline{V})$ is a finite group and moreover
$\operatorname{Aut}(\overline{V})=\operatorname{Lin}(\overline{V})$ except possibly when $(n,d)=(3,4)$, see \cite[Theorem 1,2]{matsumura1963automorphisms}.

\begin{defn}
A smooth projective variety $V$ defined over a field $k$ is called a \emph{smooth
$L$-hypersurface over $k$ of degree $d$ in $\mathbb{P}^n_{\overline{k}}$} and $L$ is a \emph{hypersurface model-field of definition for $V$}, where $L/k$ is a field extension inside $\overline{k}$, if the base extension $V\otimes_kL$
is $L$-isomorphic to a non-singular hypersurface model $H_{V\otimes_kL,d,n}:F_{V\otimes_kL}(X_0,\ldots,X_n)=0$ of degree $d$ with coefficients in $L$. For the special case $L=k$, $V$ is simply called a \emph{smooth
hypersurface over $k$}.
\end{defn}
Suppose that $V$ is a smooth $\overline{k}$-hypersurface over $k$ of
degree $d\geq4$ in $\mathbb{P}^n_{\overline{k}}$. Hence, by finiteness and linearity of $\operatorname{Aut}(V\otimes_k\overline{k})$ for $(n,d)\neq(3,4)$, we get that $\overline{V}:=V\otimes_k\overline{k}$ has a linear series, that
allows us to embed $\Upsilon:\overline{V}\hookrightarrow
\mathbb{P}^n_{\overline{k}}$ as a smooth hypersurface. Further, this
linear system is unique up to
$\operatorname{PGL}_{n+1}(\overline{k})$-conjugation. Thus, for $(n,d)\neq(3,4)$, we can always think
$\operatorname{Aut}(\overline{V})$ as a finite subgroup of
$\operatorname{PGL}_{n+1}(\overline{k})$, leaving invariant a fixed
non-singular hypersurface model
$H_{\overline{V},d,n}:F_{\overline{V}}(X_0,\ldots,X_n)=0$ of degree
$d\geq4$ coming from the embedding $\Upsilon:\overline{V}\hookrightarrow
\mathbb{P}^n_{\overline{k}}$ over $\overline{k}$. In other words,
any other non-singular hypersurface model over $\overline{k}$ is
defined by an equation of the form
$F_{P^{-1}\overline{V}}(X_0,\ldots,X_n):=F_{\overline{V}}(P(X_0,\ldots,X_n))=0$
for some $P\in \text{PGL}_{n+1}(\overline{k})$.

The aim of this paper is to make a study for fields of definition of
non-singular hypersurface models of a smooth
$\overline{k}$-hypersurface $V$ over $k$, also for its twists, by
considering the embedding
$\text{Aut}(\overline{V})\hookrightarrow\text{PGL}_{n+1}(\overline{k})$.
We note that if the smooth projective variety $V$, or any of its
twists over $k$, is a smooth hypersurface over $k$, then we have an
embedding of $\operatorname{Gal}(\overline{k}/k)$-groups for its
automorphism group into $\text{PGL}_{n+1}(\overline{k})$. This
approach leads to two natural questions: the first one, given a
smooth projective variety $V$ defined over a field $k$ and admitting
a non-singular $\overline{k}$-hypersurface model, does it have a
non-singular hypersurface model over $k$?; and secondly, if the
answer is yes, does every twist of $V$ over $k$ also have a
non-singular hypersurface model over $k$? For both questions the
answer is \textbf{No} in general, it does not. We obtain results for
the varieties for which the above questions always have an
affirmative answer, and we show different examples concerning the
negative general answer.

The paper is a generalization of our work in \cite{badr2016twists} jointly with Elisa Lorenzo Garc\'{\i}a,
where the same problem was addressed, but for smooth hypersurfaces
in $\mathbb{P}^2$, that is for smooth plane curve of degree $d\geq
4$.

\subsection*{Acknowledgments}
The authors would like to express their gratitude to David Kohel who
suggested to us the generalization of the work in
\cite{badr2016twists} for hypersurfaces instead of plane curves,
also to Joaquim Ro\'e and Xavier Xarles for their useful comments
and suggestions.
\end{section}
\begin{section}{Statements of the results}


First, we study the minimal field $L$ where there exists a
non-singular model over $L$ for a smooth $\overline{k}$-hypersurface
$V$ defined over $k$.

We show the following, which follows from
\cite{roe2014galois}.
\begin{thm}\label{sufficientconditions}
Let $V$ be a smooth $\overline{k}$-hypersurface over a perfect field
$k$ of degree $d\geq4$ in $\mathbb{P}^n_{\overline{k}}$ such that $(n,d)\neq(3,4)$. Then, $V$ is not necessarily a
smooth hypersurface over $k$. However, it does in any of the
following situations:
 \begin{enumerate}[(i)]
   \item if $V$ has $k$-rational points, i.e. when $V(k)\neq\emptyset$,
   \item if $\operatorname{gcd}(d,n+1)=1$,
   \item if the $(n+1)$-torsion $\operatorname{Br}(k)[n+1]$ of the Brauer group $\operatorname{Br}(k)$ is trivial.
 \end{enumerate}
 In general, $V$ has a non-singular hypersurface model over a field extension $L/k$ of degree $[L:k]\,|\,n+1$. Also, in case of number fields $k$, we show in \S 4.1 an example of a smooth $\overline{k}$-hypersurface over $k$, which is not a smooth hypersurface over $k$, but it does over a Galois field extension of degree $n+1$ over $k$.
 \end{thm}
\subsection*{Notation and conventions}
We write $\text{Gal}(L/k)$ for the Galois group of the extension $L/k$, and we consider left actions. The Galois cohomology sets of a
$\text{Gal}(L/k)$-group $G$ when $L/k$ is Galois are denoted by
$\text{H}^i(\text{Gal}(L/k), G)$ with $i\in\{0,1\}$ respectively.
For the particular case $L=\overline{k}$, we use $\operatorname{G}_k$ instead of
$\text{Gal}(\overline{k}/k)$ and $\text{H}^1(k,G)$ instead of
$\text{H}^1(\operatorname{Gal}(\overline{k}/k),G)$.

Second, we assume that $V$ is a smooth hypersurface over $k$. We
obtain the next Theorem characterizing the twists of $V$, which are
also smooth hypersurfaces over $k$. In particular, we prove the
following.
\begin{thm}\label{lem1} Let $V$ be a smooth hypersurface over a perfect field $k$ identified with a fixed non-singular hypersurface model $H_{V,d,n}:F_{V}(X_0,\ldots,X_n)=0$, with $F_V(X_0,\ldots,X_n)\in k[X_0,\ldots,X_n]$ with $(n,d)\neq(3,4)$.
Then, there exists a natural
map
$$\Sigma: \operatorname{H}^1(\operatorname{Gal}(\overline{k}/k),\operatorname{Aut}(\overline{V}))\rightarrow
\operatorname{H}^1(\operatorname{Gal}(\overline{k}/k),\operatorname{PGL}_{n+1}(\overline{k})),$$
defined by the inclusion
$\operatorname{Aut}(H_{V,d,n}\otimes_k\overline{k})\subseteq\operatorname{PGL}_{n+1}(\overline{k})$
as $\operatorname{Gal}(\overline{k}/k)$-groups. The preimage
$\Sigma^{-1}([\mathbb{P}^n_k])$ is formed by the set of all twists
of $V$ over $k$ that are smooth hypersurfaces over $k$, where
$[\mathbb{P}^n_k]$ denotes the class of the trivial Brauer-Severi
variety of dimension $n$ over $k$. Moreover, any such twist is
obtained through an automorphism of $\mathbb{P}^{n}_{\overline{k}}$,
that is, the twist is $k$-isomorphic to
$$H_{F_{M^{-1}V,d,n}}:\,F_{M^{-1}V}(X_0,\ldots,X_n):=F_{V}(M(X_0,\ldots,X_n))\in k[X_0,\ldots,X_n],$$
for some $M\in \operatorname{PGL}_{n+1}(\overline{k})$.

We can reinterpret the map $\Sigma$ as the map sending a twist $V'$ over $k$ to the Brauer-Severi variety $B$ where it lives (cf. \cite[Lemma 5]{roe2014galois}).
\end{thm}
Then, we have assertions similar to those in Theorem \ref{sufficientconditions}.
\begin{cor}\label{lem200}
Let $V$ be a smooth hypersurface over a perfect field $k$ of degree
$d\geq4$ inside $\mathbb{P}^n_{\overline{k}}$ with $(n,d)\neq(3,4)$. The map
$\Sigma$ in Theorem \ref{lem1} is trivial, if
$\operatorname{gcd}(d,n+1)=1$ or $\operatorname{Br}(k)[n+1]$ is
trivial. In particular, for such situations, any twist $V'$ for $V$
over $k$ is also a smooth hypersurface over $k$.

On the other hand, we construct in \S 4.4 a one parameter family
$H_{F_a,d=2p,n}$, for $a\in k$, of smooth hypersurfaces over a number field
$k_{d}$ where $p$ is an odd prime integer, and we show
that each member of the family has a twist over $k_d$ that does not
admits a hypersurface model over $k_d$.
\end{cor}

Finally, we study the twists for a smooth hypersurface $V$ over
$k$, such that $\text{Aut}(\overline{V})$ is a cyclic group.
\begin{defn} Let $V/k:F_V(X_0,\ldots,X_n)=0$ be a smooth hypersurface over a perfect field $k$, where $F_V(X_0,\ldots,X_n)\in k[X_0,\ldots,X_n]$.
We call a twist $V'$ for $V$ over $k$, a \emph{diagonal twist}, if there
exists an $M\in \text{PGL}_{n+1}(k)$ and a diagonal matrix $D\in\text{PGL}_{n+1}(\overline{k})$, such that $V'$ is $k$-isomorphic to
$$F_{(MD)^{-1}V}(X_0,\ldots,X_n):=F_{V}(MD(X_0,\ldots,X_n))=0,$$
where $F_{(MD)^{-1}V}(X_0,\ldots,X_n)\in k[X_0,\ldots,X_n]$
\end{defn}
We prove the following.
\begin{thm}[Diagonal twists]\label{lem411} Let $V/k:F_V(X_0,\ldots,X_n)=0$ be a smooth hypersurface over a perfect field $k$ of degree $d\geq4$ with $(n,d)\neq(3,4)$.
Assume that
$\operatorname{Aut}(V\otimes_k\overline{k})\hookrightarrow\operatorname{PGL}_{n+1}(\overline{k})$, given by the linear system, is a
    non-trivial cyclic group of order $m$ generated by $\psi=\operatorname{diag}(1,\zeta_m^{a_1},\zeta_m^{a_2},\ldots,\zeta_m^{a_n})$ for some $a_i\in\mathbb{N}$, where $\zeta_m$ denotes a fixed primitive $m$-th root of unity provided that $m$ is coprime to the characteristic of $k$.
     Then, all the twists in
    $\operatorname{Twist}_k(V)$ are diagonal given over $k$ by a non-singular polynomial equation of the form $F_{D^{-1}V}(X_0,\ldots,X_n)=0$ where
    $F_{D^{-1}V}(X_0,\ldots,X_n)\in k[X_0,\ldots,X_n]$ and $D$ is a diagonal matrix in $\operatorname{PGL}_{n+1}(\overline{k})$.
    In particular, the map $\Sigma$ in Theorem \ref{lem1} is trivial.
\end{thm}

\begin{rem}\label{lem412} Let $V/k$ be a smooth hypersurface over a perfect field $k$
of characteristic $p\geq 0$, and identify it with a non-singular
hypersurface model $H_{V,d,n}:F_{V}(X_0,\ldots,X_n)=0$ over $k$ with
$d\geq 4$ and $(n,d)\neq(3,4)$. Suppose also that
$\operatorname{Aut}(H_{V,d,n}\otimes_k\overline{k})\subseteq
\operatorname{PGL}_{n+1}(\overline{k})$ is a cyclic group of order
$n$, generated by a matrix $\psi$ whose conjugacy class in
$\operatorname{PGL}_{n+1}(k)$ does not contain elements of diagonal
shapes. Then, the twists of $V$ over $k$ whose image under $\Sigma$
is trivial (i.e., the ones that are smooth hypersurfaces over $k$),
are expected not to be represented by diagonal twists. For example,
for $n=2$ in \cite{badr2016twists}, we provide a smooth plane curve
with a cyclic non-diagonal automorphism group in the above sense,
where not all of its twists are diagonal. To construct such examples
in higher dimensional $\mathbb{P}^n$, it requires a knowledge about
the structure of automorphism groups and also the Twisting Theory
for smooth hypersurfaces living there.
\end{rem}
\end{section}
\begin{section}{Brauer-Severi varieties and central simple algebras}
Let $U$ be a quasi-projective variety defined over a perfect field
$k$. A twist for $U$ is a variety $U'$ defined over $k$ that is
isomorphic over $\overline{k}$ to $U$, but not necessarily over $k$.
A twist $U'$ is called trivial if $U$ and $U'$ are $k$-isomorphic.
The set of all twists of $U$ modulo $k$-isomorphisms is denoted by
$\operatorname{Twist}_k(U)$. It is well-known that the set
$\operatorname{Twist}_k(U)$ is in one-to-one correspondence with the
first Galois cohomology set
$\operatorname{H}^1(k,\operatorname{Aut}(U\otimes_k\overline{k}))$
given by $[U]\mapsto \xi:\tau\mapsto\xi_{\tau}:=\phi\circ\,
^{\tau}\phi^{-1},$ for $\tau\in G_k$, where
$\phi:U'\otimes_k\overline{k}\rightarrow U\otimes_k\overline{k}$ is
a fixed $\overline{k}$-isomorphism, see \cite[Chp. III ]{MR1324577}.
\subsection*{Brauer-Severi varieties} A Brauer-Severi variety $\mathcal{B}$ of dimension $n$ over a perfect field $k$ is simply a twist of the $n$-dimensional projective space $\mathbb{P}^n_k$ over $k$. A field extension $L/k$ is said to be a splitting field of $\mathcal{B}$, if $\mathcal{B}\otimes_k L\simeq\mathbb{P}^n_L$, and we say that $L/k$ splits $\mathcal{B}$ or that $\mathcal{B}$ splits over $L$.

In particular, we obtain:
\begin{cor}\label{301}
The set $\operatorname{Twist}_k(\mathbb{P}^n_k)$ is in bijection with $\operatorname{H}^1(k,\operatorname{Aut}(\mathbb{P}^n_{\overline{k}}))=\operatorname{H}^1(k,\operatorname{PGL}_{n+1}(\overline{k}))$.
\end{cor}
Moreover, we have:
\begin{thm}[Severi, Ch\^{a}telet, Lichtenbaum]\label{302}
Let $\mathcal{B}$ be a Brauer-Severi variety of dimension $n$ over a
perfect field $k$. Then, there exists a field extension $L/k$ of
degree $[L:k]\,|\,n+1$ such that $L$ splits $\mathcal{B}$. Moreover,
$\mathcal{B}$ splits over $k$, if it has $k$-rational points or
contains a hypersurface of degree relatively prime with $n+1$.
\end{thm}

\begin{proof}
The result is due to Severi (cf. \cite[X, \S6, Excercise 1]{MR0354618}), Ch\^{a}telet (cf. his PhD thesis \cite{MR3533116}), and Lichtenbaum (cf. \cite[Theorem 5.4.10]{MR2266528}). One also can read the proof of \cite[Theorem 5]{roe2014galois}.
\end{proof}

We conclude from \cite[Lemma 4]{roe2014galois}
 the following.
\begin{thm}[Ro\'e-Xarles]\label{proprx} Let $V$ be a smooth projective variety over a perfect field $k$. Suppose that, for some fixed $n\geq 2$, there is a unique (modulo automorphisms) $n$-dimensional
 linear series over $\overline{k}$ invariant under the $\operatorname{Gal}(\overline{k}/k)$-action, giving a morphism $h:V\otimes_k\overline{k}\rightarrow\mathbb{P}^n_{\overline{k}}$. Then, there exists
a Brauer-Severi variety $\mathcal{B}$ of dimension $n$ defined over
$k$, together with a $k$-morphism $g:V\hookrightarrow \mathcal{B}$,
such that
$g\otimes_k\overline{k}:V\otimes_k\overline{k}\rightarrow\mathbb{P}^n_{\overline{k}}$
equals to $h$.
\end{thm}
\subsection*{Central simple algebras} A central simple algebra over a field $k$ is a finite dimensional 
associative 
algebra over $k$, which is simple, i.e contains no non-trivial (two sided) ideal and the multiplication operation is not uniformly zero, 
and for which the center 
is exactly $k$.

A field extension $L/k$ is said to be a splitting field of a central simple algebra $\mathcal{A}$ over $k$, if $A\otimes_k L\simeq M_n(L)$ for some $n$, and we say that $L/k$ splits $\mathcal{A}$.

\begin{example}[Cyclic algebras]\label{parametrization1}
Let $L/k$ be a cyclic extension of degree $n+1$ with $\operatorname{Gal}(L/k)=\langle\sigma\rangle$.
 Then, an element of $\operatorname{H}^1(\operatorname{Gal}(L/k),\operatorname{PGL}_{n+1}(L))$ represented by a $1$-cocycle $f:\operatorname{Gal}(L/k)\rightarrow\operatorname{PGL}_{n+1}(L)$ is completely determined by the value of $f(\sigma)$, which is subject to
$$
f(\sigma)\cdot\,^{\sigma}(f(\sigma))\cdot\,^{\sigma^2}(f(\sigma))\cdot\ldots\cdot\,^{\sigma^{n}}(f(\sigma))=1.
$$
For instance, let $a\in k^*$ and consider the matrices
$$C_a:=\left(
         \begin{array}{ccccc}
           0 & 0 & \ldots & 0 & a \\
           1 & 0 & \ldots & 0 & 0 \\
           0 & 1 & \ldots & 0 & 0 \\
           \vdots &  &  &  &  \\
           0 & 0 & \ldots & 1 & 0 \\
         \end{array}
       \right); D_a:=\left(
         \begin{array}{cccccc}
           0 & 1 & 0&\ldots & 0 & 0 \\
           0 & 0 & 1&\ldots & 0 & 0 \\
           0 & 0 & 0&\ldots & 0 & 0 \\
           \vdots &\vdots&  &  &  &  \\
           0&0&\ldots&\ldots&0&1\\
           a & 0 & \ldots &\ldots& 0 & 0 \\
         \end{array}
       \right).$$

Define a $1$-cocycle $f$ by setting $f(\sigma)=C_a\,\modu L^*$. Hence
\begin{eqnarray*}
f(\sigma)\cdot\,^{\sigma}(f(\sigma))\cdot\,^{\sigma^2}(f(\sigma))\cdot\ldots\cdot\,^{\sigma^{n}}(f(\sigma))=C_a^{n+1}\,\modu L^*=aI\,\modu L^*=I\,\modu L^*,
\end{eqnarray*}
where $I$ is the identity matrix. According to \cite[Theorem
5.4]{tengan2009central}, we can associate to this $1$-cocycle
central simple algebra $\mathcal{A}$ over $k$ of dimension $(n+1)^2$
that splits by $L$, by considering the set of matrices $M\in
M_{n+1}(L)$ satisfying $C_a\,^{\sigma}M\,C_a^{-1}=M$. One finds that
the matrices $I,C_a,\ldots,C_a^{n}\in\mathcal{A}$ as well as
$$S_b:=\operatorname{diag}(b,\sigma(b),\ldots,\sigma^{n}(b)),\,\,\text{for}\,\,b\in
L.$$ Therefore, $\bigoplus S_bC_a^i$ is a $k$-subalgebra of the
correct dimension $(n+1)^2$, and corresponds to the algebra
$\mathcal{A}$ defined by the $1$-cocycle $f$. This kind of
$k$-algebras are also known as the \emph{cyclic algebra $(\chi,a)$}
associated to $a\in k^*$ and the character
$\chi:\text{Gal}(L/k){\stackrel{\simeq}{\longrightarrow}}\Z/(n+1)\Z$
sending $\chi(\sigma)=-1\modu (n+1)$.

In the above computations, we may replace $C_a$ with $D_a$ and we
get symmetrically the cyclic algebra $(\chi,a)$ associated to $a\in
k^*$ and the character
$\chi:\text{Gal}(L/k){\stackrel{\simeq}{\longrightarrow}}\Z/(n+1)\Z$
sending $\chi(\sigma)=1\modu (n+1)$, since $C_a
S_b=\,^{\sigma^{-1}}S_b\, C_a$ is changed to $D_a S_b=\,^{\sigma}S_b
\, D_a$,

For the complete details, we refer to \cite[Example 5.5]{tengan2009central}.
\end{example}

\begin{thm}[J. H. Maclagan-Wedderburn, R. Brauer]\label{wedderburnthm100}
Given a central simple algebra $\mathcal{A}$ over $k$, there exists (up to isomorphism) a unique division algebra $\mathcal{D}$ with center $k$ and a positive integer $n$, such that $\mathcal{A}$ is isomorphic to $M_n(\mathcal{D})$. Consequently, the dimension $\operatorname{dim}_k(\mathcal{A})$ of $\mathcal{A}$ over $k$ is always a square.
\end{thm}
Theorem \ref{wedderburnthm100} gives a strict relation between
central simple algebras and division algebras, and suggests the
introduction of the following equivalence relation: Two central
simple algebras $\mathcal{A}_1$ and $\mathcal{A}_2$ over the same
field $k$ are equivalent if there are positive integers $m$, $n$
such that $M_m(\mathcal{A}_1)\simeq M_n(\mathcal{A}_2)$.
Equivalently, $\mathcal{A}_1$ and $\mathcal{A}_2$ are equivalent if
$\mathcal{A}_1$ and $\mathcal{A}_2$ are matrix algebras over a
division algebra (up to isomorphism).
\begin{defn}
The set of all Brauer equivalence classes of central simple algebras
over $k$ equipped with the tensor product of $k$-algebras is an
abelian group (cf. \cite[Proposition 2.4.8]{MR2266528}), known as
the Brauer group of $k$ and is denoted by $\operatorname{Br}(k)$.
The \emph{period} of a central simple algebra over $k$ is defined to
be its order as an element of the Brauer group. The $m$-torsion
$\operatorname{Br}(k)[m]$ of the Brauer group $\operatorname{Br}(k)$
is the set of all elements of $\operatorname{Br}(k)$ of order dividing $m$.
\end{defn}
Recall that each Brauer equivalence class contains a unique (up to
isomorphism) division algebra. Define the \emph{index} of a central
simple algebra to be the degree of the division algebra $D$ that is
Brauer equivalent to it, i.e. the square root of the dimension of $D$ over $k$.

We particularly have:
\begin{cor}[cf. \cite{MR2060023}]\label{periodindex}
The period of a central simple algebra over $k$ divides its index, and hence is finite.
\end{cor}

\subsection*{Interplay} We find in the literature several approaches to the connection between central simple algebras and Brauer-Severi varieties; First, the connection between quaternion algebras and plane conics observed by E. Witt in \cite{MR1545510}. In its general form, we mention for example the most elementary
one promoted by J.-P. Serre in his books \cite{MR0354618, MR1324577}. The main observation
is that central simple algebras of dimension $(n+1)^2$ over a perfect field $k$ as well as $n$-dimensional Brauer-Severi varieties over $k$ can both be described by classes in one and the same cohomology set $\operatorname{H}^1(k,\operatorname{PGL}_{n+1}(\overline{k}))$.

\end{section}

\begin{section}{Proofs of the results}
\subsection{A smooth $\overline{k}$-hypersurface not having a non-singular hypersurface model over $k$}

Fix an algebraic closure $\overline{\mathbb{Q}}$ of $\mathbb{Q}$,
and let $k\subset\overline{\mathbb{Q}}$ be a number field. Suppose
that
$$f(t)=t^{n+1}+\lambda_nt^n+\ldots+\lambda_1t+(-1)^{n+1}\lambda_0=\prod_{i=0}^{n}(t-a_i)\in k[t]$$
is an irreducible polynomial of degree $n+1\geq 3$ over $k[t]$,
whose splitting field $k_f$ over $k$ satisfies
$\operatorname{Gal}(k_f/k)\cong\mathbb{Z}/(n+1)\mathbb{Z}$, that is, the inverse Galois problem for cyclic group
$\mathbb{Z}/(n+1)\mathbb{Z}$ is realizable, (recall that the inverse
Galois problem is true for any solvable group over a number field by
the contribution of Shafarevich in 1954). Fix a generator
$$\sigma:a_0\rightarrow a_1\rightarrow
a_2\rightarrow\ldots\rightarrow a_{n-1}\rightarrow a_n\rightarrow
a_0$$ for the Galois group $\operatorname{Gal}(k_f/k).$ Take a
positive integer $d$ not relatively prime with $n+1$, such that the
following holds: there exists $\beta\in k^*$, which is not a norm in
$k_f$, such that
$\lambda_0^{n+1}\beta^d=\prod_{i=0}^{n}\sigma^i(\alpha)$,
equivalently $\beta^d=N_{k_f/k}(\alpha/\lambda_0)$, for some
$\alpha\in k_f^*\setminus k^*$. Next, we define a smooth
hypersurface over $k_f$ by the equation
$$
H_{f,\alpha,d,n}:\lambda_0X_0^d+\sum_{i=1}^{n}\left(\frac{1}{\lambda_0^{i-1}}\prod_{j=0}^{i-1}\sigma^j(\alpha)\right)X_i^d=0,
$$
of degree $d\geq4$ where $(n,d)\neq(3,4)$. The matrix
$$\phi_{\sigma}:=\left(
         \begin{array}{ccccc}
           0 & 0 & \ldots & 0 & \beta \\
           1 & 0 & \ldots & 0 & 0 \\
           0 & 1 & \ldots & 0 & 0 \\
           \vdots &  &  &  &  \\
           0 & 0 & \ldots & 1 & 0 \\
         \end{array}
       \right)
$$
defines an isomorphism $\phi_{\sigma}=C_{\beta}:
H_{f,\alpha,d}\rightarrow \,^{\sigma}H_{f,\alpha,d,n}$ (equivalently, $\phi_{\sigma}^{-1}=D_{\beta^{-1}}:\,^{\sigma}H_{f,\alpha,d,n}\rightarrow
H_{f,\alpha,d,n}$), which satisfies the Weil's condition of decent
    \cite{MR0082726}
    ($\phi_{\sigma^{n+1}}=\phi_{\sigma}^{n+1}=1$). We therefore obtain that the variety is defined over $k$, and that there exists
    an isomorphism $\varphi_0:V_{k}\rightarrow H_{f,\alpha,d,n}$
    where $V_{k}$ is a rational model such that $\psi_{\sigma}=\phi^{-1}_{\sigma}=\varphi_0\circ\,^{\sigma}\varphi_{0}^{-1}\in \text{PGL}_{n+1}(k)$.
The assignation
$\psi_{\tau}:=\varphi_0\circ\,^{\tau}\varphi_{0}^{-1}$ defines an
element of $\text{H}^1(\text{Gal}(k_f/k),\text{PGL}_{n+1}(k_f))$,
corresponding to a cyclic algebra which is non-trivial because
$\beta$ is not a norm of an element of $k_f$ (cf. \cite[\S
2.1]{MR2396201}). Consequently, $\varphi_0$ is not given by an
element of $\text{PGL}_{n+1}(k_f)$, or even
$\text{PGL}_{n+1}(\overline{k}_f)$, since the cohomology class by
the inflation map is not trivial. Therefore, the variety $V_k$ is
not a smooth hypersurface over $k$ (otherwise, $V_k$ is identified
via a $k$-isomorphism $\psi_0$ with a non-singular hypersurface
model defined over $k$. Thus, by \cite[Theorem 1,2]{matsumura1963automorphisms}, the cohomology class $[\varphi_0\circ\psi_0]$ is represented
by an $M\in\operatorname{PGL}_{n+1}(\overline{k}_f)$, which is not
since $[\varphi_0]\neq1$). As a concrete example, we precise the
above construction in $\mathbb{P}^2,\mathbb{P}^3$ and
$\mathbb{P}^4$.

\begin{enumerate}
  \item In $\mathbb{P}^2;$ take $k=\mathbb{Q}$ and consider the irreducible polynomial $f(t)=t^3+12t^2-64$ over $k$ (thus $\lambda_0=64$).
  As we can check with SAGE \cite{sage}, the discriminant of the field $k_f$ is a power of $3$, and the prime $2$ becomes inert in $k_f$, hence is not a norm in $k_f$. Consequently,
  we can assume, for example, that $d=9m-12, \beta=2$ and $\alpha={a_0}8^{m}$ with $m\geq 2$ an integer.
 \item In $\mathbb{P}^3;$ take $k=\Q$ and consider the irreducible polynomial $f(t)=t^4+t^3+2t^2-4t+3$
  over $k$ (thus $\lambda_0=3$). Then, $k_f$ is a
  cyclic extension of $\mathbb{Q}$ with Galois group isomorphic to $\Z/4\Z$ (we can think about it inside $\mathbb{Q}(\zeta_{13})$), moreover, $\lambda_0=3$
  splits completely in $k_f$, so it may be a norm. In the ring of integers $\mathcal{O}_{k_f}$ of $k_f$, $\beta=17=\mathfrak{p}_1\mathfrak{p}_2$, where $\mathfrak{p}_1$ and $\mathfrak{p}_1$ are not principal ideals of $\mathcal{O}_{k_f}$, also one of the generators of $\mathfrak{p}_1$, say $\gamma$, belongs to
  $k_f\setminus \mathbb{Q}$. It is easy to check that
  $N_{k_f/\mathbb{Q}}(\gamma)=17^2$, and we can choose
  $\alpha:=3\cdot 17^m \gamma \in k_f\setminus \mathbb{Q}$ with
  $m\geq 2$, giving degrees $d=4m-2$.
  \item In $\mathbb{P}^4$; the polynomial $f(t)=t^5-t^4-4t^3+3t^2+3t-1$ is irreducible over $\mathbb{Q}$. Its field of decomposition $\Q_f$ is cyclic over $\mathbb{Q}$ of degree $5$ (we can think $\Q_f=\Q(\cos(2\pi/11))\subseteq\mathbb{Q}(\zeta_{11})$). In the ring of integers $\mathcal{O}_{\Q_f}$ of $\Q_f$, the torsion units are $\pm 1$, and the roots $\alpha_i$ of $f$ are units in $\mathcal{O}_{\Q_f}$. Suppose that $m>1$ is an integer satisfying $\operatorname{gcd}(m,n+1=5)=1$ and
$\operatorname{gcd}(\varphi(m),5)=1$, where $\varphi$ is the Euler function. Now,
  take $d=5m$ and $k=\mathbb{Q}(\zeta_d)$ with $\zeta_d$ a fixed $d$-th primitive root of unity inside $\overline{\Q}$. 
      We still have that $f(t)$ is irreducible over $k=\mathbb{Q}(\zeta_d)$, since $k\cap\mathbb{Q}_f=\Q$.
      Also, $k_f$ does not contain torsion roots of unity more than $\langle\zeta_d\rangle$, therefore $\zeta_d$ is not a norm form $k_f$ to $k$. 
      In particular, we can particulary set $\beta=\zeta_d$ and
      $\alpha=a_0$.
\end{enumerate}

\begin{cor}
Let $V$ be a smooth $\overline{k}$-hypersurface over a perfect field
$k$ of degree $d\geq4$ in $\mathbb{P}^n_{\overline{k}}$ such that $(n,d)\neq(3,4)$. Then, $V$ is not
necessarily a smooth hypersurface over $k$.
\end{cor}
\subsection{Minimal fields of definition for non-singular hypersurface models}

In this subsection, $V$ is a smooth $\overline{k}$-hypersurface of degree $d\geq4$ in
$\mathbb{P}^n_{\overline{k}}$ with $n\geq 2$ and $(n,d)\neq(3,4)$. Accordingly,
$\overline{V}:=V\otimes_k\overline{k}$ has an $n$-dimensional linear
series over $\overline{k}$ that allows us to embed
$\Upsilon:\overline{V}\hookrightarrow \mathbb{P}^n_{\overline{k}}$
as a non-singular hypersurface, and such linear series is unique
modulo conjugation in $\operatorname{PGL}_{n+1}(\overline{k})$.

We first show:
\begin{prop}
Let $V$ be a smooth $\overline{k}$-hypersurface over $k$ of degree
$d\geq4$ in $\mathbb{P}^n_{\overline{k}}$, where $n\geq2$ and
$(n,d)\neq(3,4)$. There exists a non-singular
hypersurface model over a field extension $L/k$ of degree
$[L:k]\,|\,n+1.$
\end{prop}

\begin{proof}
We know form \cite[Lemma 4]{roe2014galois} that the set of
$k$-morphisms (modulo automorphisms) to some $n$-dimensional
Brauer-Severi varieties over $k$ are in bijection with the
base-point free $n$-dimensional linear series over $\overline{k}$
which is invariant under the $G_k$-action. Consequently, we may
apply Theorem \ref{proprx} to obtain a $k$-morphism
$g:V\hookrightarrow \mathcal{B}$ to a Brauer-Severi variety
$\mathcal{B}$ of dimension $n$ over $k$, such that
$g\otimes_k\overline{k}:\overline{V}\hookrightarrow\mathbb{P}^n_{\overline{k}}$
equals to $\Upsilon$. Moreover, by the virtue of \cite[Theorem
13-(5)]{roe2014galois}, there exists a field extension $L/k$ of
index $[L:k]\,|\,n+1$ that splits $\mathcal{B}$ (this means that
$\mathcal{B}\otimes_kL$ is $L$-isomorphic to $\mathbb{P}^n_L$).
Hence, we reduce to an embedding of $V\otimes_kL$ into
$\mathbb{P}^n_L$ as the smooth variety $g(V)\otimes_kL$. By
assumption, $g(V)\otimes_k\overline{k}$ is a hypersurface inside
$\mathbb{P}^n_{\overline{k}}$, then so does
$g(V)\otimes_k\overline{L}\subset\mathbb{P}^n_{\overline{L}}$.
Consequently, $g(V)\otimes_kL$ has dimension $n-1$ and therefore it
is a non-singular hypersurface model for $V\otimes_kL$ over $L$.
\end{proof}

Second, we show:
\begin{prop}\label{prop3}
Let $V$ be a smooth $\overline{k}$-hypersurface over $k$ of degree
$d\geq4$ in $\mathbb{P}^n_{\overline{k}}$, where $n\geq2$ and
$(n,d)\neq(3,4)$. Then, $V$ is a smooth hypersurface
over $k$, if $V(k)\neq\emptyset,\,\operatorname{gcd}(d,n+1)=1$, or
$\operatorname{Br}(k)[n+1]$ is trivial.
 \end{prop}

\begin{proof}
Using \cite[Theorem 13-(1),(2)]{roe2014galois}, we obtain that the base field $k$ splits $\mathcal{B}$ when $V(k)\neq\emptyset$ or $\operatorname{gcd}(d,n+1)=1$. On the other hand, let $[\mathcal{A}]$ be the image of $\mathcal{B}$ in the Brauer group $\operatorname{Br}(k)$ (in particular, $\mathcal{B}$ splits over a field extension $L/k$ if and only if $\mathcal{A}$ does over $L/k$). Due to Ch\^{a}telet in his thesis, the division algebra associated to $\mathcal{A}$ has dimension dividing $n+1$. That is, $\mathcal{B}$ corresponds to an element of the $(n+1)$-torsion of $\operatorname{Br}(k)$, since the order of a central simple algebra as an element of $\operatorname{Br}(k)$ divides its index, which is the square root of the dimension of the associated division algebra (see Corollary \ref{periodindex}). Therefore, $\mathcal{B}$ also splits over $k$, if $\operatorname{Br}(k)[n+1]$ is trivial, being associated to a trivial central simple algebra over $k$.

By the above discussion, we can see that our variety $V$ is living inside a trivial Brauer-Severi variety in any of the prescribed situations. This in turns gives a non-singular hypersurface model over $k$, and we conclude.
\end{proof}

\begin{cor}
Let $V$ be a smooth $\overline{k}$-hypersurface over $k$ of degree
$d\geq4$ in $\mathbb{P}^n_{\overline{k}}$, where $n\geq2$ and
$(n,d)\neq(3,4)$. Then, $V$ is a smooth hypersurface
over $k$ if
\begin{enumerate}
  \item $k$ is an algebraically closed field;
  \item $k$ is a finite field;
  \item $k$ is the function field of an algebraic curve over an algebraically closed field;
\item $k$ is an algebraic extension of $\mathbb{Q}$, containing all roots of unity;
\item $k=\mathbb{R}$ and $n+1$ odd.
\end{enumerate}
  \end{cor}

 \begin{proof}
 In the following cases, every division algebra over a field $k$ is $k$ itself, so that the Brauer group $\operatorname{Br}(k)$ is trivial.
\begin{itemize}
  \item $k$ is an algebraically closed field (cf. \cite[Example 4.2]{tengan2009central}).
  \item $k$ is a finite field (Wedderburn's Little theorem, cf. \cite[page 162]{MR554237}).
  \item $k$ is the function field of an algebraic curve over an algebraically closed field (Tsen's Theorem, cf. \cite[Theorem 6.2.8]{MR2266528}). More generally, the Brauer group vanishes for any quasi-algebraically closed field.
\item $k$ is an algebraic extension of $\mathbb{Q}$, containing all roots of unity (cf. \cite[page 162]{MR554237}).
\end{itemize}
Finally, there are just two non-isomorphic real division algebras with center $\mathbb{R}$: $\mathbb{R}$ itself and the quaternion algebra $\mathbb{H}$. Since $\mathbb{H}\otimes\mathbb{H}\simeq M_4(\mathbb{R})$, the class of $\mathbb{H}$ has order two in the Brauer group. That is, the Brauer group $\operatorname{Br}(\mathbb{R})$ is the cyclic group of order two and then, for $n+1$ odd, $\operatorname{Br}(\mathbb{R})[n+1]$ does not contain non-trivial elements of order dividing $n+1$.

Now, the result is an immediate consequence of Proposition \ref{prop3}.
 \end{proof}

\end{section}
\subsection{Twists which are smooth hypersurfaces over the base field $k$}
 Now, assume that $V$ is a smooth hypersurface over $k$, in particular $V$ is $k$-isomorphic to a non-singular hypersurface model of degree $d\geq 4$ of the form
$$H_{F_V,d,n}:F_V(X_0,\ldots,X_n):=F_{M^{-1}\overline{V}}(X_0,\ldots,X_n)=F_{\overline{V}}(M(X_0,\ldots,X_n))\in k[X_0,\ldots,X_n],$$
for some $M\in\operatorname{PGL}_{n+1}(\overline{k})$ with $n\geq2$ and $(n,d)\neq(3,4)$. Since
$\operatorname{Aut}(H_{F_V,d,n}\otimes_k\overline{k})\subset\operatorname{PGL}_{n+1}(\overline{k})$
as $G_k$-groups, then $\operatorname{Aut}(\overline{V})$ is
naturally embedded into $\operatorname{PGL}_{n+1}(\overline{k})$ as
$G_k$-groups, and we get a well-defined map
$$\Sigma:\operatorname{Twist}_k(V)=\operatorname{H}^1(k,\operatorname{Aut}(\overline{V}))\rightarrow
\operatorname{H}^1(k,\operatorname{PGL}_{n+1}(\overline{k})).$$

\begin{proof}[Proof of Theorem \ref{lem1} and Corollary \ref{lem200}]
Let $[\mathbb{P}^n_k]$ denotes the class of the trivial
Brauer-Severi variety of dimension $n$ over $k$ in
$\operatorname{Twist}_k(\mathbb{P}^n_k)=
\operatorname{H}^1(k,\operatorname{PGL}_{n+1}(\overline{k}))$. If a
twist $V'/k$ is $k$-isomorphic to a non-singular hypersurface model
$F_{V'}(X_0,\ldots,X_n)=0$ over $k$, then $F_{V'}(X_0,...,X_n)=0$ and
$F_{V}(X_0,\ldots,X_n)=0$ are isomorphic through some $M'\in
\text{PGL}_{n+1}(\overline{k})$ by \cite[Theorem 1,
2]{matsumura1963automorphisms} when $n\geq 3$ and is well-known when $n=2$ in case of plane curves of degree $\geq4$. Hence, the corresponding $1$-cocycle $\sigma\mapsto
M'\circ\,^{\sigma}(M')^{-1}\in\operatorname{Aut}(H_{F_V,d,n}\otimes_k\overline{k})$
is trivial in
$\operatorname{H}^1(k,\operatorname{PGL}_{n+1}(\overline{k}))$,
being cohomologous to the trivial $1$-cocycle.
Conversely, if $\Sigma([V'])$ is trivial (that is $V'/k$ is living
inside a trivial Brauer-Severi variety of dimension $n$ over $k$),
then it must be given by a $\overline{k}$-isomorphism
$\varphi:\{F_{V}(X_0,\ldots,X_n)=0\}\rightarrow V'$ induced by some
$\tilde{M}\in\text{PGL}_{n+1}(\overline{k})$, i.e. $V'$ is
$k$-isomorphic to $F_{(\tilde{M})^{-1}V}(X_0,\ldots,X_n)=0$. This
would give a non-singular hypersurface model over $k$ for $V'$.

In general, any twists $V'$ for $V$ over $k$ is again a smooth $\overline{k}$-hypersurface over $k$ in $\mathbb{P}^n_{\overline{k}}$ of degree $d$. It only remains to apply Theorem \ref{sufficientconditions} for $V'$ to conclude that $V'$ is a smooth hypersurface over $k$ if $\operatorname{gcd}(d,n+1)=1$ or if $\operatorname{Br}(k)[n+1]$ is trivial. In particular, $\Sigma$ is the trivial map in both cases, which was to be finally shown.
\end{proof}

\subsection{A smooth hypersurface over number field $k$, having a twist which is not a smooth hypersurface over $k$}
Let $p$ be an odd prime number and fix $\zeta_{p}$, a primitive
$p$-th root of unity inside $\overline{\mathbb{Q}}$. Assume that
$H_{F,d,n}:F(X_0,\ldots,X_n)=0\subset\mathbb{P}^n_{\overline{k}}$ is a
smooth hypersurface of degree $d=2p$, not relatively prime with
$n+1$, defined over the number field $k=\mathbb{Q}(\zeta_{p})$.
Suppose further that the projective matrix
$$\phi:=\left(
         \begin{array}{ccccc}
           0 & 0 & \ldots & 0 & \zeta_{p}\\
           1 & 0 & \ldots & 0 & 0 \\
           0 & 1 & \ldots & 0 & 0 \\
           \vdots &  &  &  &  \\
           0 & 0 & \ldots & 1 & 0 \\
         \end{array}
       \right)
$$
induces an automorphism of $H_{F,d=2p,n}\otimes_k\overline{k}$. For
example, the following family parameterized by $a\in k$:
$$H_{F_a,d=2p,n}:F_a(X_0,\ldots,X_n)=\sum_{i=0}^{n}X_i^d+\sum_{i=0}^{n-1}aX_i^{p}\sum_{j=i+1}^{n}X_j^{p}=0.$$
Next, take $m\in k^*\setminus (k^*)^p$, then $x^p-m$ is irreducible
in $k[x]$ by a theorem of Abel, the Galois field extension
$L_m=k(\sqrt[p]{m})$ has Galois group
$\operatorname{Gal}(L/k)=\langle\sigma\rangle\simeq\Z/p\Z$, where
$\sigma(\sqrt[p]{m})=\zeta_{p}\sqrt[p]{m}$. Define the $1$-cocycle
by
$$\xi:\sigma\mapsto\phi\in\operatorname{H}^1(\operatorname{Gal}(L_m/k),\operatorname{PGL}_{n+1}(L_m)).$$
Since no new primitive root of unity appears in $L_m$ than $k$,
$\zeta_{p}$ is not a norm in $L_m$, $[\xi]$ is non-trivial in
$\operatorname{H}^1(\operatorname{Gal}(L_m/k),\operatorname{PGL}_{n+1}(L_m))$
by \cite[\S1.1]{MR2396201}, and hence the image of $\xi$ in
$\operatorname{H}^1(\operatorname{Gal}(\overline{k}/k),\operatorname{PGL}_{n+1}(\overline{k}))$,
which coincides with the inflation of $\sigma\mapsto\phi$, is not
trivial. Consequently, it corresponds to a twist $V'$ for
$H_{F_a,d,n}$ over $k$ living inside a non-trivial Brauer-Severi
variety of dimension $n$ over $k$ (that is,
$\Sigma([V'])\neq[\mathbb{P}^n_k]$).

\subsection{Diagonal twists}
Let $V/k:F_V(X_0,\ldots,X_n)=0$ be a smooth hypersurface over a perfect field $k$.
Assume that $\operatorname{Aut}(V\otimes_k\overline{k})\hookrightarrow\operatorname{PGL}_{n+1}(\overline{k})$ is a
    non-trivial cyclic group of order $m$ (relatively prime with the characteristic of $k$), generated by $\psi=\operatorname{diag}(1,\zeta_m^{a_1},\zeta_m^{a_2},\ldots,\zeta_m^{a_n})$ for some $a_i\in\mathbb{N}$, where $\zeta_m$
    is a fixed primitive $m$-th root of unity in $\overline{k}$.
\begin{proof}[Proof of Theorem \ref{lem411} and \ref{lem412}]
It suffices to notice that the embedding $\operatorname{Aut}(V\otimes_k\overline{k})\hookrightarrow
\text{PGL}_{n+1}(\overline{k})$ factors through $\operatorname{GL}_{n+1}(\overline{k})$. Thus
    the map $\Sigma$ in Theorem \ref{lem1} factors as follows:
    $$
    \Sigma: \,\operatorname{H}^1(\operatorname{Gal}(\overline{k}/k),\operatorname{Aut}(V\otimes_k\overline{k}))\rightarrow\operatorname{H}^1(\operatorname{Gal}(\overline{k}/k),\operatorname{GL}_{n+1}(\overline{k}))\rightarrow\operatorname{H}^1(\operatorname{Gal}(\overline{k}/k),\operatorname{PGL}_{n+1}(\overline{k})).
    $$
    Moreover, $\operatorname{H}^1(k,\operatorname{GL}_{n+1}(\overline{k}))=1$ by Hilbert's 90 Theorem, so the map $\Sigma$ is trivial. By Theorem \ref{lem1} any twist has a non-singular plane model $F_{P^{-1}V}(X_0,\ldots,X_n)=0$ over $k$, for some $P\in\operatorname{PGL}_{n+1}(\overline{k})$. Since $P\circ\,^{\sigma}(P^{-1})\in
\operatorname{Aut}(V\otimes_k\overline{k})=\langle\operatorname{diag}(1,\zeta_m^{a_1},\zeta_m^{a_2},\ldots,\zeta_m^{a_n})\rangle$ for any $\sigma\in\operatorname{Gal}(\overline{k}/k)$, 
then $^{\sigma}P=P\,\circ\,\operatorname{diag}(1,v_1,\ldots,v_n)$ for some $m$th roots of unity
$v_i$. Writing $P=(a_{i,j})$, one easily
deduces that $\sigma(a_{i,j})=v_j a_{i,j}$ for all $i,j$. Hence, for any fixed integer $j$, we have
$\sigma(a_{i,j})a_{i,j}^{-1}=\sigma(a_{i',j})a_{i',j}^{-1}$. That is, $a_{i,j}a_{i',j}^{-1}$ is $\operatorname{Gal}(\overline{k}/k)$-invariant, which in turns gives that $a_{i,j}=m_ia_{i',j}$ for some $m_i\in k$. In particular, $P$ reduces to $MD$ for some $D$ a diagonal projective $(n+1)\times(n+1)$ matrix and
$M\in\text{PGL}_{n+1}(k)$. This proves that all the twists are diagonal. However, the non-singular hypersurface model $F_{(MD)^{-1}V}(X_0,\ldots,X_n)=0$ over $k$ is $k$-isomorphic through $M$ to $F_{D^{-1}V}(X_0,\ldots,X_n)=0$. Consequently, $F_{D^{-1}V}(X_0,\ldots,X_n)=0$ defines a non-singular hypersurface model over $k$ for the twist.

\end{proof}


\bibliographystyle{amsalpha}

\end{document}